\documentclass{amsart}
\usepackage{amssymb}
\usepackage{graphicx}

\newtheorem{theorem}{Theorem}

\newtheorem{corollary}{Corollary}
\newtheorem*{thmA}{Theorem A}
\newtheorem*{thmB}{Theorem B}
\newtheorem*{thmC}{Theorem C}
\theoremstyle{definition}
\newtheorem{remark}{Remark}

\allowdisplaybreaks

\author{Mario Krni\'{c}  and Nicu\c{s}or Minculete}
\date{}
\title[Bounds for the $p$-angular distance]{Bounds for the  $p$-angular distance and characterizations of  inner product spaces}

\begin{document}

\renewcommand{\thetheorem}{\arabic{theorem}}

\setcounter{theorem}{0}

\begin{abstract}
\noindent Based on a suitable improvement  of a triangle inequality, we derive new mutual bounds for  $p$-angular distance $\alpha_p[x,y]=\big\Vert \Vert x\Vert^{p-1}x- \Vert y\Vert^{p-1}y\big\Vert$, in a normed linear space $X$. We show that our
estimates are more accurate than the previously known upper bounds established by Dragomir, Hile and Maligranda. Next, we give several characterizations  of inner product spaces
with regard to the $p$-angular distance. In particular, we prove that if $|p|\geq |q|$, $p\neq q$, then $X$ is an inner product space if and only if for every $x,y\in X\setminus \{0\}$,
$${\alpha_p[x,y]}\geq \frac{{\|x\|^{p}+\|y\|^{p} }}{\|x\|^{q}+\|y\|^{q}  }\alpha_q[x,y].$$
\end{abstract}



\maketitle

\section{Introduction}

Throughout this paper, $\left(X, \left\Vert \cdot \right\Vert \right)$ stands for a real normed  linear space.
During decades, the problem of providing necessary and sufficient conditions for a normed space to be an inner product space
has been studied by numerous authors. Some recent characterizations of inner product spaces can be found in \cite{app2, alrashed, alsina, amir, app1, rasias, vuvu} and the  references therein.

One of the most celebrated
characterizations of  inner product spaces has been based on the so-called Dunkl-Wiliams inequality which provides an upper bound for an angular distance (or Clarkson distance, see \cite{klarkson}) between nonzero vectors $x,y\in X$, defined by
$$
\alpha[x,y]=\left\Vert\frac{x}{\Vert x \Vert}- \frac{y}{\Vert y\Vert}\right\Vert.
$$
Namely, Dunkl and Wiliams \cite{dunkl}, proved that the  inequality
\begin{equation*}
\alpha[x,y]\leq \frac{4\left\Vert x-y\right\Vert }{\left\Vert
x\right\Vert + \left\Vert y\right\Vert }
\end{equation*}
holds for all nonzero vectors $x,y$ in a real normed linear space $X$. The authors also showed that  constant $4$ can be replaced by $2$ if $X$
is an inner product space. In addition,
Kirk and Smiley \cite{stenli}, proved that the  inequality
\begin{equation}\label{stanley}
\alpha[x,y]\leq\frac{2\left\Vert x-y\right\Vert }{\left\Vert
x\right\Vert + \left\Vert y\right\Vert }
\end{equation}
serves as a characterization of an inner product space.

In 1993, Al-Rashed \cite{alrashed}, generalized  (\ref{stanley}) by showing that
if $0<q\leq 1$, then a normed linear space $X=\left(X, \left\Vert \cdot \right\Vert \right)$ is an inner product space if and only if the  inequality
\begin{equation}\label{al-rashed}
\alpha[x,y]\leq\frac{2^{\frac{1}{q}}\left\Vert x-y\right\Vert }{\left(\Vert
x\Vert^{q} + \Vert y\Vert^{q} \right)^{\frac{1}{q}}}
\end{equation}
holds for all nonzero vectors $x,y\in X$. Moreover, in 2010, Dadipour and  Moslehian \cite{dadipur}, extended (\ref{al-rashed}) by giving characterization of
an inner product  space expressed in terms of $p$-angular distance $\alpha_p[x,y]$, defined by
$$
\alpha_p[x,y]=\left\Vert \Vert x\Vert^{p-1}x- \Vert y\Vert^{p-1}y\right\Vert,\ p\in \mathbb{R},
$$
where $x,y$ are nonzero vectors (note that $\alpha_0[x,y]=\alpha[x,y]$). Finally, Rooin et al. \cite{rocky}, extended the above  results from \cite{alrashed, dadipur}
by giving a whole series of comparative relations for the $p$-angular distance which serve as characterizations of inner product spaces. We give here only a part
of the main result from \cite{rocky} that will be important in our study.
\begin{thmA}
Let $X$ be a normed linear space, $\mathrm{dim}\ \! X\geq 3$, and $p,q\in \mathbb{R}$ with $|p|\leq |q|$ and $p\neq q$. Then the following statements are mutually equivalent:
\begin{enumerate}
\item[(I)] $X$ is an inner product space.

\item[(II)] If $0\leq \frac{p}{q}<1$, then for any $x,y\neq 0$,
\begin{equation*}
\alpha_p[x,y]\leq \frac{2\alpha_q[x,y]}{\|x\|^{q-p}+\|y\|^{q-p}}.
\end{equation*}

\item[(III)] For any $x,y\neq 0$,
\begin{equation*}
\frac{\alpha_p[x,y]}{\|x\|^{\frac{p}{2}}\|y\|^{\frac{p}{2}} }\leq \frac{\alpha_q[x,y]}{\|x\|^{\frac{q}{2}}\|y\|^{\frac{q}{2}} }.
\end{equation*}

\item[(IV)] If $p\in \mathbb{R}\setminus\{0\}$, then for any $x,y\neq 0$,
\begin{equation*}
\alpha_{-p}[x,y]=\|x\|^{-p}\|y\|^{-p}\alpha_p[x,y].
\end{equation*}

\item[(V)] If $0\leq \frac{p}{q}<1$, then there exist $r\in \mathbb{R}$ such that for any $x,y\neq 0$,
\begin{equation*}
\alpha_p[x,y]\leq \frac{2^{\frac{1}{r}}\alpha_q[x,y]}{\left(\|x\|^{r(q-p)}+\|y\|^{r(q-p)}\right)^{\frac{1}{r}}}.
\end{equation*}

\item[(VI)] If $0< \frac{p}{q}<1$, then for any $x,y\neq 0$ with $\|x\|\neq \|y\|$,
\begin{equation*}
\alpha_p[x,y]< \frac{\alpha_q[x,y]}{\min\{\|x\|^{q-p}, \|y\|^{q-p} \}}.
\end{equation*}
\end{enumerate}
\end{thmA}
Some interesting characterizations of inner product spaces connected with the concept of the  $p$-angular distance can also be found in \cite{amini-mia,degen, krnic,rooin0, rooin}.
In particular, Rooin et. al. \cite{rooin}, established a very nice characterization of an inner product space by giving an explicit formula for the $p$-angular distance. More precisely,
 they proved that if $p\neq 1$, then a normed space $X$ is an inner product space if and only if
 \begin{equation}\label{zap}
\alpha_p^2[x,y] =\Vert x \Vert^{p-1}\Vert y \Vert^{p-1}\Vert x-y \Vert^{2}+\left( \Vert x \Vert^{p-1} -\Vert y \Vert^{p-1}\right)\left( \Vert x \Vert^{p+1} -\Vert y \Vert^{p+1}\right)
\end{equation}
holds for all nonzero elements $x,y\in X$.

In the last few decades, several authors have been interested in providing bounds for the $p$-angular distance in an arbitrary normed linear space.
In 2006, Maligranda \cite{maligranda}, established the following mutual bounds for angular distance $\alpha[x,y]$, based on improved form of the  triangle inequality:
\begin{equation}\label{obostranoao}
\dfrac{\Vert x-y\Vert-|\Vert x\Vert - \Vert y\Vert|}{\min\{\Vert x\Vert,\Vert y\Vert\}}\leq \alpha[x,y]\leq \dfrac{\Vert x-y\Vert+|\Vert x\Vert - \Vert y\Vert|}{\max\{\Vert x\Vert,\Vert y\Vert\}}.
\end{equation}
The author also established the  following upper bounds for the $p$-angular distance, dependent on index $p$:
\begin{equation}\label{erstemali}
\alpha_p[x,y]\leq \left\{\begin{array}{lr}
p \max \!^{p-1}\{\|x\|, \|y\|\}\|x-y\|, &  p\geq 1, \\
{(2-p)\|x-y\|}{\max \!^{p-1}\{\|x\|, \|y\|\}}, &  0\leq p<1,\\
 (2-p) \frac{\min \!^{p}\{\|x\|, \|y\|\}}{\max\{\|x\|,\|y\| \}}\|x-y\|       ,& p<0.
\end{array} \right.
\end{equation}
We also recall the following upper bound for the $p$-angular distance stated by Hile \cite{hile}:
\begin{equation}\label{obicni hile}
\alpha_p[x,y]\leq \frac{\|y\|^{p}-\|x\|^{p}}{\|y\|-\|x\|}\|x-y\|,
\end{equation}
 for $p\geq 1$ and  $x,y\in X$ with $\|x\|\neq \|y\|$.

In 2015,  Dragomir \cite{ukraine}, established the  integral upper bounds for the $p$-angular distance, which resulted in a refinement of
the Maligranda inequality (\ref{erstemali}), when $p\geq 2$. The author also showed that the Hile inequality (\ref{obicni hile}) is more precise
than the first inequality in (\ref{erstemali}),  when $p\geq 2$. In 2018, Rooin et al. \cite{rooin}, generalized the above results from \cite{ukraine,hile,maligranda}
in the sense that the above bounds for the $p$-angular distance were expressed in terms of the $q$-angular distance.

 In addition, some mutual bounds for the  $p$-angular distance based on the  triangle inequality have been derived by Dragomir in \cite{dragomirMIA}, which will be discussed in the next section.

The main objective of the present paper is to establish some new bounds for the  $p$-angular distance and to provide new characterizations of inner product spaces.
After this introduction, in Section \ref{sec2} we establish mutual bounds for the $p$-angular distance, based on a suitable refinement of the triangle inequality. In particular, we show that our
estimates are more accurate than those in (\ref{erstemali}), (\ref{obicni hile}), and  the corresponding ones in \cite{dragomirMIA}.
In Section \ref{sec3} we provide some characterizations of inner product spaces which are not covered by Theorem A. First, we establish the inequality that seems familiar to the Hile inequality (\ref{obicni hile}),
although it shows significantly different behavior since it characterizes an inner product space. Finally, we  give several characterizations of an inner product space which are established through the connection between
mutual bounds for $\alpha[x,y]$ given by (\ref{obostranoao}), and the explicit formula for the  $p$-angular distance (\ref{zap}), when $p=0$. It should be noted here that our results are mainly inspired by recent papers \cite{rocky, rooin}.

\section{More accurate bounds for the  $p$-angular distance in a normed linear space}\label{sec2}

The main objective of this section is to establish some new mutual bounds for the $p$-angular distance in an arbitrary normed linear space, and to compare them
with some previously known results. Our new estimates are based on a suitable improvement of the  triangle inequality.

We start our discussion with the  lower bounds that correspond to inequalities in  (\ref{erstemali}).
Namely, Rooin et al. \cite{rooin}, derived  mutual bounds for angular distance $\alpha_p[x,y]$, expressed in terms of angular distance
$\alpha_q[x,y]$, by using a similar method as it has been done in \cite{maligranda}. Consequently, they obtained the  lower bounds for the $p$-angular distance
that correspond to inequalities in (\ref{erstemali}). However, the lower bounds for  the  $p$-angular distance can be established simply by transforming
inequalities in (\ref{erstemali}). Therefore, for the reader's convenience, we establish lower bounds that correspond to (\ref{erstemali}), with an
alternative proof.
\begin{corollary}[Rooin et al. \cite{rooin}]\label{korolarcic}
Let $X=\left(X, \left\Vert \cdot \right\Vert \right)$ be a  normed linear space. Then the inequalities
\begin{equation}\label{malkonz1}
\alpha_p[x,y]\geq \left\{\begin{array}{lr}
\frac{p}{2p-1} \max \!^{p-1}\{\|x\|, \|y\|\}\|x-y\|, &  p\geq 1, \\
{p}{\max \!^{p-1}\{\|x\|, \|y\|\}}\|x-y\|, &  0\leq p<1,\\
  \frac{p}{2p-1} \frac{\min \!^{p}\{\|x\|, \|y\|\}}{\max\{\|x\|,\|y\| \}}\|x-y\|       ,& p<0,
\end{array} \right.
\end{equation}
hold for all  nonzero vectors $x,y\in X$.
\end{corollary}
\begin{proof}[An alternative proof:]  Let $p\geq 1$. Rewriting the second inequality in (\ref{erstemali}) with $\frac{1}{p}$ instead of $p$,
we obtain
$$
\alpha_\frac{1}{p}[x,y]\leq \frac{2p-1}{p}\cdot\frac{\|x-y\|}{\max^{1-\frac{1}{p}}\{ \|x\|, \|y\| \}}.
$$
Now, since
 $$\alpha_\frac{1}{p}[\|x\|^{p-1}x,\|y\|^{p-1}y]=\big\|\|x\|^{p(\frac{1}{p}-1)}\|x\|^{p-1}x-\|y\|^{p(\frac{1}{p}-1)}\|y\|^{p-1}y\big\|=\|x-y\|,$$
  substituting $x'=\|x\|^{p-1}x$ and $y'=\|y\|^{p-1}y$ instead of $x$ and $y$, respectively,  in the last inequality, we obtain
$$
\|x-y\|\leq \frac{2p-1}{p}\cdot\frac{\alpha_p[x,y]}{\max^{p-1}\{ \|x\|, \|y\| \}},
$$
which yields the first inequality  in (\ref{malkonz1}).

The second inequality in (\ref{malkonz1}) holds trivially for $p=0$.
 To prove that the second inequality  in (\ref{malkonz1}) holds for $0<p<1$, we utilize the first inequality in (\ref{erstemali}) rewritten with $\frac{1}{p}$ instead of $p$, and with $x'=\|x\|^{p-1}x$ and $y'=\|y\|^{p-1}y$ instead of $x$ and $y$, respectively.
 Similarly to above, we obtain
 $$
 \|x-y\|\leq \frac{\max^{1-p}\{\|x\|,\|y\|\}}{p}\alpha_p[x,y],
 $$
 which yields the  second inequality in (\ref{malkonz1}). In the same way, the third inequality in (\ref{malkonz1}) follows from the third inequality in (\ref{erstemali}) rewritten with parameter $\frac{1}{p}$ and with vectors
 $x'=\|x\|^{p-1}x$ and $y'=\|y\|^{p-1}y$ instead of $x$ and $y$, respectively.
\end{proof}

\begin{remark}
By  the same method as in the proof of Corollary \ref{korolarcic}, inequalities in (\ref{erstemali}) and (\ref{malkonz1}) can be transformed into mutual bounds expressed in terms of
$\alpha_q[x,y]$, as it has been done in \cite{rooin}. For example, consider the first inequality in (\ref{malkonz1}) with $\frac{p}{q}\geq 1$ instead of $p$:
$$
\alpha_\frac{p}{q}[x,y]\geq \frac{p}{2p-q} \max \!^{\frac{p}{q}-1}\{\|x\|, \|y\|\}\|x-y\|.
$$
Now, with $x''=\|x\|^{q-1}x$ and $y''=\|y\|^{q-1}x$, instead of $x$ and $y$ respectively, the above inequality transforms to
$$
\alpha_{p}[x,y]\geq \frac{p}{2p-q}\max\{\|x\|^{p-q}, \|y\|^{p-q}  \}\alpha_q[x,y],
$$
since $\alpha_{p/q}[x'',y'']=\alpha_p[x,y]$. For more details about this class of inequalities the reader is referred to \cite{rooin}.
 Obviously, the last inequality is equivalent to the first inequality in (\ref{malkonz1}), so in this section we will only consider
inequalities in (\ref{erstemali}) and (\ref{malkonz1}).
\end{remark}

Now, our goal is to derive more accurate  bounds for the  $p$-angular distance than those in (\ref{erstemali}) and (\ref{malkonz1}). We start with establishing  a suitable refinement and reverse of the triangle inequality.

\begin{theorem}\label{teoremcic}
Let  $X=\left(X, \left\Vert \cdot \right\Vert \right)$ be a  normed linear space and let $f:\mathbb{R}^{+}\rightarrow \mathbb{R}$ be a nonnegative function. Then the inequalities
\begin{equation}
\begin{split}\label{minimamaksima}
&\  f\left(\| x \|\right)\|x\|+f\left(\| y \|\right)\|y\|-\left(\|x\|+\|y\|-\|x-y\|  \right)\max\{f\left(\| x \|\right), f\left(\| y \|\right)\}\\
\leq &\  \left\|f\left(\| x \|\right)x-f\left(\| y \|\right)y \right\|\\
\leq & \ f\left(\| x \|\right)\|x\|+f\left(\| y \|\right)\|y\|-\left(\|x\|+\|y\|-\|x-y\|  \right)\min\{f\left(\| x \|\right), f\left(\| y \|\right)\}
\end{split}
\end{equation}
hold for all nonzero vectors $x,y\in X$.
\end{theorem}

\begin{proof}
Without loss of generality, we can suppose that $f\left(\| x \|\right)\geq f\left(\| y \|\right)$. In this case, the left inequality in (\ref{minimamaksima})
reduces to
$$
f\left(\| y \|\right)\|y\|-f\left(\| x \|\right)\|y\|+\|f\left(\| x \|\right)x-f\left(\| x \|\right)y\|\leq \left\|f\left(\| x \|\right)x-f\left(\| y \|\right)y \right\|.
$$
Clearly, since $f\left(\| x \|\right)\geq f\left(\| y \|\right)$, the last inequality is equivalent to
$$
\|f\left(\| x \|\right)x-f\left(\| x \|\right)y\|\leq \left\|f\left(\| x \|\right)x-f\left(\| y \|\right)y \right\|+\left\| f\left(\| y \|\right)y-f\left(\| x \|\right)y \right\|,
$$
which  represents the triangle inequality. Similarly, the second inequality in (\ref{minimamaksima}) reduces to
$$
 \left\|f\left(\| x \|\right)x-f\left(\| y \|\right)y \right\|\leq  f\left(\| x \|\right)\|x\|- f\left(\| y \|\right)\|x\|+\| f\left(\| y \|\right)x-f\left(\| y \|\right)y \|,
$$
which is again equivalent to triangle inequality
$$
 \left\|f\left(\| x \|\right)x-f\left(\| y \|\right)y \right\|\leq  \| f\left(\| x \|\right)x-f\left(\| y \|\right)x \|+\| f\left(\| y \|\right)x-f\left(\| y \|\right)y \|.
$$
The proof is now complete.
\end{proof}

With a suitable choice of a nonnegative function $f$, the
 previous theorem provides mutual bounds for the $p$-angular distance.

\begin{corollary}\label{korolarcicjos}
Let  $X=\left(X, \left\Vert \cdot \right\Vert \right)$ be a  normed linear space and let $p\geq 1$. Then the inequalities
\begin{equation}
\begin{split}\label{alfa-p}
&\  \|x\|^{p}+\|y\|^{p}-\left(\|x\|+\|y\|-\|x-y\|  \right)\max\! {}^{p-1}\{\|x\|,\|y\|\}\\
\leq &\  \alpha_p[x,y]\\
\leq & \ \|x\|^{p}+\|y\|^{p}-\left(\|x\|+\|y\|-\|x-y\|  \right)\min\! {}^{p-1}\{\|x\|,\|y\|\}
\end{split}
\end{equation}
hold for all nonzero elements $x,y\in X$. If $p<1$, then the inequalities in {\rm(\ref{alfa-p})} are reversed.
\end{corollary}

\begin{proof}
It follows from Theorem \ref{teoremcic} by putting $f(t)=t^{p-1}$ and noting  that $f$ is increasing (decreasing) for $p\geq 1$ ($p<1$).
\end{proof}

\begin{remark}
If $p=0$, then, taking into account obvious relations
\begin{equation*}
\begin{split}
2\min\{\Vert x\Vert,\Vert y\Vert\}=&\Vert x\Vert + \Vert y\Vert-|\Vert x\Vert - \Vert y\Vert|,\\
2\max\{\Vert x\Vert,\Vert y\Vert\}=&\Vert x\Vert + \Vert y\Vert+|\Vert x\Vert - \Vert y\Vert|,
\end{split}
\end{equation*}
Corollary \ref{korolarcicjos} provides  mutual bounds for angular distance $\alpha[x,y]$, given by (\ref{obostranoao}).
\end{remark}

In 2009, Dragomir \cite{dragomirMIA}, established the following two pairs of upper and lower bounds for the  $p$-angular distance:
\begin{equation}\label{dragomirove-gornje}
\alpha_p[x,y]\leq \left\{\begin{array}{lr}
D_{\geq 1}(p), &  p\geq 1, \\
D_{<1}(p), &  p<1,
\end{array} \right.\quad \mathrm{and} \qquad \alpha_p[x,y]\leq \left\{\begin{array}{lr}
S_{\geq 1}(p), &  p\geq 1, \\
S_{<1}(p), &  p<1,
\end{array} \right.
\end{equation}
where
\begin{equation*}
\begin{split}
D_{\geq 1}(p)&=\|x-y\|\min{}^{p-1}\{\|x\|,\|y\|\}+\big|\|x\|^{p-1} -\|y\|^{p-1} \big|\max\{\|x\|,\|y\|\},\\
D_{<1}(p)&=\|x-y\|\max{}^{p-1}\{\|x\|,\|y\|\}+\big|\|x\|^{p-1} -\|y\|^{p-1} \big|\max\{\|x\|,\|y\|\},\\
S_{\geq 1}(p)&=\|x-y\|\max{}^{p-1}\{\|x\|,\|y\|\}+\big|\|x\|^{p-1} -\|y\|^{p-1} \big|\min\{\|x\|,\|y\|\},\\
S_{< 1}(p)&=\|x-y\|\min{}^{p-1}\{\|x\|,\|y\|\}+\big|\|x\|^{p-1} -\|y\|^{p-1} \big|\min\{\|x\|,\|y\|\},
\end{split}
\end{equation*}
and
\begin{equation}\label{dragomirove-donje}
\alpha_p[x,y]\geq \left\{\begin{array}{lr}
d_{\geq 1}(p), &  p\geq 1, \\
d_{<1}(p), &  p<1,
\end{array} \right.\quad \mathrm{and} \qquad \alpha_p[x,y]\geq \left\{\begin{array}{lr}
s_{\geq 1}(p), &  p\geq 1, \\
s_{<1}(p), &  p<1,
\end{array} \right.
\end{equation}
where
\begin{equation*}
\begin{split}
d_{\geq 1}(p)&=\|x-y\|\max{}^{p-1}\{\|x\|,\|y\|\}-\big|\|x\|^{p-1} -\|y\|^{p-1} \big|\max\{\|x\|,\|y\|\},\\
d_{< 1}(p)&=\|x-y\|\min{}^{p-1}\{\|x\|,\|y\|\}-\big|\|x\|^{p-1} -\|y\|^{p-1} \big|\max\{\|x\|,\|y\|\},\\
s_{\geq 1}(p)&=\|x-y\|\min{}^{p-1}\{\|x\|,\|y\|\}-\big|\|x\|^{p-1} -\|y\|^{p-1} \big|\min\{\|x\|,\|y\|\},\\
s_{< 1}(p)&=\|x-y\|\max{}^{p-1}\{\|x\|,\|y\|\}-\big|\|x\|^{p-1} -\|y\|^{p-1} \big|\min\{\|x\|,\|y\|\}.
\end{split}
\end{equation*}
The estimates in (\ref{dragomirove-gornje}) and (\ref{dragomirove-donje}) were also  established as a consequences of  certain variants of triangle inequalities,
and are given here in a more suitable form. In addition, those estimates have not been mutually compared in \cite{dragomirMIA}. Now, we will compare our bounds from Corollary \ref{korolarcicjos}
with  the above estimates. We will show that our bounds coincide with the above estimates in two cases, while in  remaining cases our estimates are more precise. Following the notation as in
(\ref{dragomirove-gornje}) and (\ref{dragomirove-donje}), the upper bounds given by (\ref{alfa-p}) and its reverse will be denoted by $K_{\geq 1}(p)$ and $K_{< 1}(p)$, while the lower bounds will be denoted by
$k_{\geq 1}(p)$ and $k_{< 1}(p)$. Here, the indices indicate the corresponding range of parameter $p$.
Note also that $K_{\geq 1}(p)=k_{< 1}(p)$, as well as $K_{< 1}(p)=k_{\geq 1}(p)$.

 \begin{theorem}
Let $x,y\in X$ be nonzero vectors. If $p\geq 1$, then holds the relation
\begin{equation*}
s_{\geq 1}(p)\leq d_{\geq 1}(p)\leq k_{\geq 1}(p)\leq \alpha_p[x,y]\leq K_{\geq 1}(p)=D_{\geq 1}(p)\leq S_{\geq 1}(p),
\end{equation*}
while for $p<1$ holds
\begin{equation*}
s_{< 1}(p)\leq d_{< 1}(p)= k_{< 1}(p)\leq \alpha_p[x,y]\leq K_{< 1}(p)\leq D_{< 1}(p)\leq S_{< 1}(p).
\end{equation*}
\end{theorem}
\begin{proof}
Without loss of generality we assume that $\|x\|\leq \|y\|$. Then, it follows that
$$D_{\geq 1}(p)=K_{\geq 1}(p)=d_{<1}(p)=k_{<1}(p)=\|y\|^{p}-\|x\|^{p-1}\|y\|+\|x-y\|\|x\|^{p-1}.$$ In addition, taking into account  the triangle inequality and  the fact that
the function $f(t)=t^{\alpha}$ is increasing (decreasing) for $\alpha\geq 0$ ($\alpha<0$), we obtain the following relations:
\begin{equation*}
\begin{split}
S_{\geq 1}(p)\!-\!D_{\geq 1}(p)=d_{\geq 1}(p)\!-\!s_{\geq 1}(p)&\!=( \|y\|^{p-1}- \|x\|^{p-1} )(\|x-y\|+\|x\|-\|y\|)\geq 0,\\
S_{< 1}(p)\!-\!D_{< 1}(p)=d_{< 1}(p)\!-\!s_{< 1}(p)&\!=( \|x\|^{p-1}- \|y\|^{p-1} )(\|x-y\|+\|x\|-\|y\|)\geq 0,
\end{split}
\end{equation*}
and
\begin{equation*}
\begin{split}
k_{\geq 1}(p)- d_{\geq 1}(p)&=(\|y\|-\|x\|)\big(\|y\|^{p-1}-\|x\|^{p-1}  \big)\geq 0,\\
D_{< 1}(p)- K_{< 1}(p)&=(\|y\|-\|x\|)\big(\|x\|^{p-1}-\|y\|^{p-1}  \big)\geq 0.
\end{split}
\end{equation*}
 The proof is now complete.
\end{proof}

\begin{remark}\label{positivity}
Although our lower bounds  $k_{\geq 1}(p)$ and $k_{< 1}(p)$ are better than the corresponding ones in (\ref{dragomirove-donje}), they are not always meaningful since they can take negative values.
Obviously,  their sign depends on parameter $p$, as well as on vectors $x,y\in X$. However, we are interested in finding values of $p$ for which $k_{\geq 1}(p)$ and $k_{< 1}(p)$ are nonnegative,
regardless of the choice of nonzero vectors $x,y\in X$.

 We first show that $k_{\geq 1}(p)\geq 0$ for $p\in [1,2]$. Without loss of generality we suppose that $\|x\|\leq\|y\|$. Then,
$$k_{\geq 1}(p)=\|x\|^{p}-\|y\|^{p-1}\|x\|+\|y\|^{p-1}\|x-y\|=\|y\|^{p-1}\varphi(p),$$
where
$$
\varphi(p)=\|x\|\left(\frac{\|x\|}{\|y\|}  \right)^{p-1}-\|x\|+\|x-y\|
$$
is a decreasing function. Then, utilizing the triangle inequality, we have
$$
\varphi(2)=\frac{\|x\|^{2}-\|x\|\|y\|+\|x-y\|\|y\|}{\|y\|}\geq \frac{\|x\|}{\|y\|}\big(\|x\|-\|y\|+\|x-y\|  \big)\geq 0.
$$
This implies that $k_{\geq 1}(p)\geq 0$ for $p\in [1,2]$. It is easy to find vectors $x,y\in X$ for which  $k_{\geq 1}(p)$ is negative, when $p>2$. To see this, let $\|x\|=n$, $\|y\|=n+1$ and $\|x-y\|=1+\frac{1}{n}$, $n\in \mathbb{N}$.
Then, $\varphi_n(p)=n\big( \frac{n}{n+1} \big)^{p-1}-n+1+\frac{1}{n}$, so by utilizing the L'Hospital rule it follows that $\lim_{n\rightarrow\infty}\varphi_n(p)=2-p<0$, which proves our assertion.

 Similarly, we show that $k_{< 1}(p)\geq 0$ for $p\in [0,1)$. Without loss of generality we suppose that $\|x\|\leq\|y\|$. Then,
$$k_{< 1}(p)=\|y\|^{p}-\|x\|^{p-1}\|y\|+\|x\|^{p-1}\|x-y\|=\|x\|^{p-1}\psi(p),$$
where
$$
\psi(p)=\|y\|\left(\frac{\|y\|}{\|x\|}  \right)^{p-1}-\|y\|+\|x-y\|
$$
is an  increasing function. We have, $\psi(0)=\|x\|-\|y\|+\|x-y\|\geq 0$, by the triangle inequality, which implies that $k_{< 1}(p)\geq 0$ for $p\in [0,1)$. Moreover, if
$x,y\in X$ are such that $\|x\|=n$, $\|y\|=n+1$, $\|x-y\|=1+\frac{1}{n}$, $n\in \mathbb{N}$, then, $\psi_n(p)=(n+1)\big( \frac{n+1}{n} \big)^{p-1}-n+\frac{1}{n}$
and $\lim_{n\rightarrow\infty}\psi_n(p)=p$, by virtue of the L'Hospital rule. This shows that if $p<0$, then $k_{< 1}(p)$ can take negative values. In conclusion, the  lower
bounds given by Corollary \ref{korolarcicjos} will be considered only for $p\in[0,2]$.
\end{remark}

The upper bounds in (\ref{dragomirove-gornje}) were roughly compared in \cite{dragomirMIA} with the Maligranda upper bounds given by  (\ref{erstemali}). In particular, it was shown  that one of the estimates
in (\ref{dragomirove-gornje}) is better  than (\ref{erstemali}) in the case when $p\geq 1$ (for more details, see \cite{dragomirMIA}).

Now, by a more precise analysis, we will make comparison of our Corollary \ref{korolarcicjos} with the Maligranda upper bounds in (\ref{erstemali}) and the lower bounds provided by Corollary \ref{korolarcic}. We will prove that the upper bounds in (\ref{alfa-p}) and its reverse are always better than the corresponding estimates in (\ref{erstemali}). In addition, we will show that if $0\leq p<1$, then the lower bound $k_{< 1}(p)$ in the reverse of (\ref{alfa-p}) is more precise than
the corresponding one in (\ref{malkonz1}).
Following the notation as above, the upper bounds in (\ref{erstemali}) will be denoted  by $M_{\geq 1}(p)$, $M_{[0,1)}(p)$, and $ M_{<0}(p)$, respectively, while the lower
bounds in (\ref{malkonz1}) will be denoted by $m_{\geq 1}(p)$, $m_{[0,1)}(p)$, and $m_{<0}(p)$.

\begin{theorem}
If $x,y\in X$ are nonzero vectors, then
\begin{equation}\label{moje boje up}
\alpha_p[x,y]\leq \left\{\begin{array}{lr}
 K_{\geq 1}(p)\leq  M_{\geq 1}(p), &  p\geq 1, \\
 K_{< 1}(p)\leq  M_{[0,1)}(p), &  0\leq p<1,\\
 K_{< 1}(p)\leq  M_{<0}(p)       ,& p<0.
\end{array} \right.
\end{equation}
 In addition, if $0\leq p< 1$, then
 \begin{equation}\label{moje boje down}
  m_{[0,1)}(p)\leq k_{< 1}(p)\leq\alpha_p[x,y].
 \end{equation}
\end{theorem}

\begin{proof}
Without loss of generality we assume that $\|x\|\leq \|y\|$. We first prove that $K_{\geq 1}(p)\leq  M_{\geq 1}(p)$, for $p\geq 1$. Taking into account our assumption,
we have
\begin{equation*}
\begin{split}
K_{\geq 1}(p)=&\ \|y\|^{p}-\|x\|^{p-1}\|y\|+\|x\|^{p-1}\|x-y\|,\\
M_{\geq 1}(p)=&\ p\|y\|^{p-1}\|x-y\|,
\end{split}
\end{equation*}
 and  we will show that $M_{\geq 1}(p)-K_{\geq 1}(p)\geq 0$, for $p\geq 1$. Since $M_{\geq 1}(p)-K_{\geq 1}(p)=\|y\|^{p-1}f(p)$, where
 \begin{equation}\label{fja}
 f(p)=p\|x-y\|-\|y\|+\left(\|y\|- \|x-y\| \right)\left( \frac{\|x\|}{\|y\|} \right)^{p-1},
 \end{equation}
 it suffices to show that $f$ is nonnegative for $p\geq 1$. We consider two cases depending on whether $\|y\|- \|x-y\|>0$ or $\|y\|- \|x-y\|\leq 0$.
 If $\|y\|- \|x-y\|\leq 0$, then $f$ is represented as a sum of two increasing functions. Consequently,  $f$
   is an   increasing function which implies that $f(p)\geq f(1)=0$, for $p\geq 1$. Otherwise, by taking a second derivative of $f$, we obtain
 $$f''(p)=\left( \|y\|- \|x-y\| \right)\big( \frac{\|x\|}{\|y\|} \big)^{p-1}\log^{2}\frac{\|x\|}{\|y\|}\geq 0.$$ Hence,  if $\|y\|- \|x-y\|>0$, then $f$ is convex on $\mathbb{R}$.
 Consequently, since $f(0)=\frac{\|y\|}{\|x\|}\left(\|y\|-\|x\|-\|x-y\|  \right)\leq 0$, by the triangle inequality, and $f(1)=0$, it follows that $f(p)\geq 0$, for all $p\geq 1$. In conclusion,
 $M_{\geq 1}(p)-K_{\geq 1}(p)\geq 0$, for all $p\geq 1$, as claimed.

 Next we show that  $K_{< 1}(p)\leq  M_{[0,1)}(p)$, for $0\leq p<1$.
 In this setting, we have
 \begin{equation*}
\begin{split}
K_{< 1}(p)=&\ \|x\|^{p}-\|x\|\|y\|^{p-1}+\|y\|^{p-1}\|x-y\|,\\
M_{[0,1)}(p)=&\ (2-p)\|y\|^{p-1}\|x-y\|.
\end{split}
\end{equation*}
Since $M_{[0,1)}(p)-K_{< 1}(p)=\|y\|^{p-1}g(p)$,
where
$$
g(p)=(1-p)\|x-y\|-\|x\|\left( \frac{\|x\|}{\|y\|} \right)^{p-1}+\|x\|,
$$
it suffices to prove that $g(p)\geq 0$ for $0\leq p<1$. It should be noted  here that $g$ is concave  on $\mathbb{R}$ since
$$g''(p)=-\|x\|\big( \frac{\|x\|}{\|y\|} \big)^{p-1}\log^{2}\frac{\|x\|}{\|y\|}\leq 0.$$ Consequently, since $g(0)=\|x-y\|+\|x\|-\|y\|\geq 0$, by the triangle inequality, and $g(1)=0$,
it follows that $g(p)\geq 0$, for all $0\leq p<1$.

The last step in connection with  the upper bounds is to show that $M_{<0}(p)-K_{< 1}(p)\geq 0$, for $p<0$. Since
$$
M_{<0}(p)=(2-p)\frac{\|x\|^{p}}{\|y\|}\|x-y\|,
$$
it follows that $M_{<0}(p)-K_{< 1}(p)=\frac{\|x\|^{p}}{\|y\|}h(p)$,
where
$$
h(p)=(2-p)\|x-y\|-\|y\|+\big(\|x\|-\|x-y\|  \big)\left(\frac{\|y\|}{\|x\|}  \right)^{p}.
$$
We will show that $h$ is decreasing function for $p<0$. We consider two cases depending on whether $\|x\|-\|x-y\|< 0$ or $\|x\|-\|x-y\|\geq 0$. If $\|x\|-\|x-y\|< 0$,
then $h$ is decreasing  on $\mathbb{R}$, since it is represented as a sum of two decreasing functions. Otherwise, if $\|x\|-\|x-y\|\geq 0$, then, by taking a derivative of $h$,
we obtain
$$
h'(p)=-\|x-y\|+\big( \|x\|-\|x-y\| \big)\left(\frac{\|y\|}{\|x\|}  \right)^{p}\log \frac{\|y\|}{\|x\|}.
$$
Our intention is to show that $h'(p)\leq 0$, for $p\leq 0$. By the Lagrange mean value theorem, there exist $\xi\in (\|x\|,\|y\|)$ such that $\log\|y\|-\log\|x\|=\frac{\|y\|-\|x\|}{\xi}$,
and consequently,
$$
\log \frac{\|y\|}{\|x\|}\leq \frac{\|y\|-\|x\|}{\|x\|}\leq \frac{\|y-x\|}{\|x\|}.
$$
Furthermore, since $\big( \frac{\|y\|}{\|x\|} \big)^{p}\leq 1$,  for $p\leq 0$, it follows that
$$
h'(p)\leq -\frac{\|x-y\|^{2}}{\|x\|},
$$
so $h$ is decreasing function for $p\leq 0$ in each case. On the other hand, since $h(0)=\|x-y\|+\|x\|-\|y\|\geq 0$, it follows that $h(p)\geq h(0)=0$, for $p<0$, which completes the proof of the third inequality
in (\ref{moje boje up}).

It remains to prove (\ref{moje boje down}).  In other words, we need to show that $k_{< 1}(p)-m_{[0,1)}(p)\geq 0$, for $0\leq p<1$. Since $k_{< 1}(p)=K_{\geq 1}(p)$ and $m_{[0,1)}(p)=M_{\geq 1}(p)$,
it follows from the first part of the proof that $m_{[0,1)}(p)-k_{< 1}(p)=\|y\|^{p-1}f(p)$, where $f$ is defined by (\ref{fja}). We have already proved that $f$
is a  convex function such that $f(0)\leq 0$ and $f(1)=0$. This implies that $f(p)\leq 0$, for all $0\leq p<1$, and consequently, $k_{< 1}(p)-m_{[0,1)}(p)\geq 0$, as claimed. The proof is now complete.
\end{proof}

\begin{remark}
Note that the previous theorem gives no answer to the question about   comparison of the  lower bounds $k_{\geq 1}(p)$ and $m_{\geq 1}(p)$, when $p\in[1,2]$. Of course, if $p=1$ these bounds coincide, otherwise they are not comparable. To see this, note that the difference
$k_{\geq 1}(p)-m_{\geq 1}(p)$, provided that $\|x\|\leq \|y\|$,  can be rewritten as $k_{\geq 1}(p)-m_{\geq 1}(p)=\|y\|^{p-2}\xi(p)$, where
$$
\xi(p)=\|x\|\left(\frac{\|x\|}{\|y\|}  \right)^{p-1}-\|x\|+\frac{p-1}{2p-1}\|x-y\|.
$$
Now, if $\|x\|=\|y\|$, then $\xi(p)\geq 0$, for all $p\in (1,2]$. On the other hand, if $\|x\|=n$, $\|y\|=n+1$, $\|x-y\|=1+\frac{1}{n}$, $n\in \mathbb{N}$, then
$\xi_n(p)=n\big( \frac{n}{n+1} \big)^{p-1}-n+\frac{p-1}{2p-1}\big(1+\frac{1}{n} \big)$, and consequently $\lim_{n\rightarrow\infty}\xi_n(p)=-\frac{2(p-1)^{2}}{2p-1}< 0$, for all $p\in (1,2]$.
This means that for each $p\in (1,2]$, we can choose $x,y\in X$ such that $k_{\geq 1}(p)> m_{\geq 1}(p)$ or $k_{\geq 1}(p)< m_{\geq 1}(p)$.

Although the lower bounds $k_{\geq 1}(p)$ and $k_{< 1}(p)$ can be negative when $p$ does not belong to    interval $[0,2]$, they are still not comparable with bounds in  (\ref{malkonz1}).
 For example,  if $\|x\|=1$, $\|y\|=4$ and $\|x-y\|=4$, then
 $m_{<0}(-1)-k_{< 1}(-1)=0.08\dot{3}>0$, while for $\|x\|=1$, $\|y\|=4$ and $\|x-y\|=4.5$, we have $m_{<0}(-1)-k_{< 1}(-1)=-0.375<0$.
\end{remark}

We have already discussed that Dragomir \cite{ukraine},  showed that the Hile inequality (\ref{obicni hile}) is more precise
than the first inequality in (\ref{erstemali}),  when $p\geq 2$ (see also \cite{rooin}). However, it turns out that our inequality in (\ref{alfa-p}) is better than the Hile inequality (\ref{obicni hile}),
for all $p\geq 1$. For the reader's convenience, the right-hand side of inequality (\ref{obicni hile}) will be denoted by $H_{\geq 1}(p)$.

\begin{theorem}
If $x,y\in X$ are nonzero vectors and $p\geq 1$, then
\begin{equation}\label{boji od hilea}
\alpha_p[x,y]\leq K_{\geq 1}(p)\leq H_{\geq 1}(p).
\end{equation}
\end{theorem}

\begin{proof}
Without loss of generality we assume that $\|x\|\leq \|y\|$. Then, considering the difference
$H_{\geq 1}(p)-K_{\geq 1}(p)$ and factoring the obtained expression, we have that
\begin{equation*}
\begin{split}
H_{\geq 1}(p)-K_{\geq 1}(p)&=\frac{\|y\|^{p}-\|x\|^{p}}{\|y\|-\|x\|}\|x-y\|-\big( \|y\|^{p}-\|x\|^{p-1}\|y\|+\|x\|^{p-1}\|x-y\|  \big)\\
&=\frac{\| y \|^{p}\|x-y\|-\| y \|^{p+1}+\| y \|^{p}\|x\|-\| x \|^{p}\|y\|}{\|y\|-\|x\|}\\
&\quad +\frac{\| x \|^{p-1}\| y \|^{2}-\| x \|^{p-1}\|y\|\|x-y\|}{\|y\|-\|x\|}\\
&=\frac{\|y\|\big(\|y\|^{p-1}-\|x\|^{p-1}  \big)\big(\|x-y\|+\|x\|-\|y\|  \big)}{\|y\|-\|x\|}.
\end{split}
\end{equation*}
Now, since $\|y\|^{p-1}-\|x\|^{p-1}\geq 0$ and $\|x-y\|+\|x\|-\|y\|\geq 0$, by the triangle inequality, it follows that
$H_{\geq 1}(p)-K_{\geq 1}(p)\geq 0$, which ensures (\ref{boji od hilea}).
\end{proof}

Theorem \ref{teoremcic} can also be exploited in deriving mutual bounds for the skew  $p$-angular distance. Recall that the skew $p$-angular distance between nonzero elements $x,y\in X$ (see \cite{rooin}) is defined by
\begin{equation*}
\beta_p[x,y]=\left\Vert \Vert y\Vert^{p-1}x- \Vert x\Vert^{p-1}y\right\Vert, \ p\in \mathbb{R}.
\end{equation*}
 When $p=0$, we set $\beta[x,y]$ for $\beta_p[x,y]$ and call it simply the  skew angular distance between $x,y\in X$. It is easy to see that the $p$-angular distance and the skew $p$-angular distance are related by
\begin{equation}\label{veza}
\beta_p [x,y]=\Vert x \Vert^{p-1}\Vert y \Vert^{p-1} \alpha_{2-p}[x,y].
\end{equation}

\begin{corollary}
Let  $X=\left(X, \left\Vert \cdot \right\Vert \right)$ be a  normed linear space and let $p\leq 1$. Then the inequalities
\begin{equation}
\begin{split}\label{beta-p}
&\ \|x\|\|y\|^{p-1}+\|x\|^{p-1}\|y\|-\frac{\|x\|+\|y\|-\|x-y\|  }{\min\! {}^{1-p}\{\|x\|,\|y\|\}}\\
\leq &\  \beta_p[x,y]\\
\leq & \ \|x\|\|y\|^{p-1}+\|x\|^{p-1}\|y\|-\frac{\|x\|+\|y\|-\|x-y\|  }{\max\! {}^{1-p}\{\|x\|,\|y\|\}}
\end{split}
\end{equation}
hold for all nonzero elements $x,y\in X$. If $p>1$, then the inequalities in {\rm (\ref{beta-p})} are reversed.
\end{corollary}

\begin{proof} We utilize inequalities in  (\ref{alfa-p}) with   $2-p$ instead of $p$, and the above formula
(\ref{veza}).
\end{proof}

\begin{remark}
Taking into account Remark \ref{positivity} and relation (\ref{veza}), it follows that the lower bounds in (\ref{beta-p}) are nonnegative for all $p\in [0,2]$. In particular,
if $p=0$, then the inequalities in (\ref{beta-p}) reduce to the following mutual bounds for   $\beta[x,y]$, obtained by
Dehghan \cite{degen}:
$$
\frac{\|x-y\|}{\min\{\|x\|, \|y\|\}}-\frac{\big|\|x\|-\|y\|  \big|}{\max\{\|x\|, \|y\|\}}\leq \beta[x,y]\leq \frac{\|x-y\|}{\max\{\|x\|, \|y\|\}}+\frac{\big|\|x\|-\|y\|  \big|}{\min\{\|x\|, \|y\|\}}.
$$
\end{remark}

In order to conclude this section, we give yet another pair of mutual bounds for the $p$-angular distance which relies  on the estimates in (\ref{obostranoao}). Namely, due to homogeneity of
angular distance $\alpha[x,y]$, we obtain mutual bounds for $\alpha_p[x,y]$, expressed in terms of $\alpha[x,y]$.
\begin{corollary}
Let $X=\left(X, \left\Vert \cdot \right\Vert \right)$ be a  normed linear space. If $p\geq 0$, then the inequalities
\begin{equation}\label{bezveze1}
\begin{split}
&\max\! {}^{p}\{\|x\|,\|y\|\}\alpha[x,y]-\left| \|x\|^{p}-\|y\|^{p} \right|\\
\leq &\  \alpha_p[x,y]\leq  \min\! {}^{p}\{\|x\|,\|y\|\}\alpha[x,y]+\left| \|x\|^{p}-\|y\|^{p} \right|
\end{split}
\end{equation}
hold for all nonzero vectors $x,y\in X$. If $p<0$, then
\begin{equation}\label{bezveze2}
\begin{split}
&\min\! {}^{p}\{\|x\|,\|y\|\}\alpha[x,y]-\left| \|x\|^{p}-\|y\|^{p} \right|\\
\leq &\  \alpha_p[x,y]\leq  \max\! {}^{p}\{\|x\|,\|y\|\}\alpha[x,y]+\left| \|x\|^{p}-\|y\|^{p} \right|.
\end{split}
\end{equation}
\end{corollary}
\begin{proof}
Rewriting (\ref{obostranoao}) with $x'=\|x\|^{p-1}x$, $y'=\|y\|^{p-1}y$ instead of $x$, $y$, respectively,  as well as noting that
 $\alpha[ax,by]=\alpha[x,y]$ when $ab>0$, we obtain
\begin{equation*}
\dfrac{\alpha_p[x,y]-|\Vert x\Vert^{p} - \Vert y\Vert^{p}|}{\min\{\Vert x\Vert^{p},\Vert y\Vert^{p}\}}\leq \alpha[x,y]\leq \dfrac{\alpha_p[x,y]+|\Vert x\Vert^{p} - \Vert y\Vert^{p}|}{\max\{\Vert x\Vert^{p},\Vert y\Vert^{p}\}}.
\end{equation*}
Consequently, the result follows from relations $\min\{\Vert x\Vert^{p},\Vert y\Vert^{p}\}=\min\! {}^{p}\{\|x\|,\|y\|\}$, for $p\geq 0$,
$\min\{\Vert x\Vert^{p},\Vert y\Vert^{p}\}=\max\! {}^{p}\{\|x\|,\|y\|\}$, for $p<0$, and the similar  formulas for maximum.
\end{proof}

\begin{remark}
The lower bound in (\ref{bezveze1}) is meaningful if the condition
\begin{equation}\label{da ne bude bezveze}
\alpha[x,y]\geq \frac{\left| \|x\|^{p}-\|y\|^{p} \right|}{\max\! {}^{p}\{\|x\|,\|y\|\}}
\end{equation}
is valid. Obviously, (\ref{da ne bude bezveze}) holds if $\alpha[x,y]\geq 1$. On the other hand,  it is easy to show that
the lower bound in (\ref{bezveze1}) can take negative values when $\alpha[x,y]<1$. For example, if $x,y\in X$ are such that $\alpha[x,y]<1$, $\|x\|=n$ and $\|y\|=n^2$, $n\in \mathbb{N}$,  then the lower bound in
(\ref{bezveze1}) is equal to $n^{2p}\big(\alpha[x,y]-1+n^{-p}\big)$, which is evidently negative for $p>0$ and sufficiently large $n$. The same conclusion can be drawn for inequality (\ref{bezveze2}),
we omit details here.
\end{remark}

\section{Bounds in inner product spaces and characterizations of inner product spaces}\label{sec3}

 Rooin et al. \cite{rocky}, proved  that if $X$ is an inner product space and  $\big| \frac{p}{q} \big|\geq 1$, then the inequality
\begin{equation}\label{apasolutno}
\alpha_p[x,y]\leq \left|\frac{\|y\|^{p}-\|x\|^{p}}{\|y\|^{q}-\|x\|^{q}}\right|\alpha_q[x,y]
\end{equation}
holds for all $x,y\in X$ such that $\|x\|\neq\|y\|$.
 On the other hand, generalizing the integral bounds for the $p$-angular distance established by Dragomir \cite{ukraine},  Rooin et al. \cite{rooin}, proved that if $\frac{p}{q}\geq 1$, then  inequality
 (\ref{apasolutno}) is valid in any normed linear space. By putting $q=1$ in (\ref{apasolutno}), we obtain the Hile inequality (\ref{obicni hile}), provided that $p>1$.

\begin{remark}
If $\frac{p}{q}\geq 1$, then inequality (\ref{apasolutno}) can be  derived directly from the Hile inequality (\ref{obicni hile}).
 Namely, assuming that (\ref{obicni hile}) holds,
it follows that holds the inequality
$$
\alpha_\frac{p}{q}[x',y']\leq \frac{\|y'\|^{\frac{p}{q}}-\|x'\|^{\frac{p}{q}}}{\|y'\|-\|x'\|}\|x'-y'\|,
$$
where  $x'=\|x\|^{q-1}x$ and $y'=\|y\|^{q-1}y$. Clearly, the last inequality reduces to (\ref{apasolutno}) since
$\alpha_{p/q}[x',y']=\alpha_p[x,y]$ and $\|x'-y'\|=\alpha_q[x,y]$.
\end{remark}

The extended Hile inequality (\ref{apasolutno}) in an inner product space
 has been established as a consequence of a generalized Dunkl-Wiliams identity between the $p$-angular and the $q$-angular distances (for more details, see \cite{rocky}).
However, this can be achieved through a simpler identity that establishes a connection between $\alpha_p[x,y]$, $p\neq 0$, and $\alpha[x,y]$.
Namely, if $p=0$, the explicit formula for the $p$-angular distance, given by (\ref{zap}), reduces to
\begin{equation}
\alpha^2[x,y] = \dfrac{\left\Vert x-y\right\Vert^2-\left(\left\Vert x\right\Vert-\left\Vert y\right\Vert\right)^2}{\left\Vert x\right\Vert\left\Vert y\right\Vert}.  \label{zadva}
\end{equation}
Then, considering (\ref{zadva}) with $x'=\|x\|^{p-1}x$ and $y'=\|y\|^{p-1}y$ instead of $x$ and $y$, respectively, as well as taking into account homogeneity of $\alpha[x,y]$, we
arrive at  the following identity that was used in \cite{krnic}:
\begin{equation}\label{identitet0}
\alpha_p^{2}[x,y]=\left( \|x\|^{p}-\|y\|^{p} \right)^{2}+\|x\|^{p}\|y\|^{p}\alpha^{2}[x,y].
\end{equation}
 Now, the
 quotient
$$
\frac{\alpha_p^{2}[x,y]}{\alpha_q^{2}[x,y]}=\frac{\left( \|x\|^{p}-\|y\|^{p} \right)^{2}+\|x\|^{p}\|y\|^{p}\alpha^{2}[x,y]}{\left( \|x\|^{q}-\|y\|^{q} \right)^{2}+\|x\|^{q}\|y\|^{q}\alpha^{2}[x,y]}
$$
can be rewritten as
\begin{equation*}
\begin{split}
&\  \alpha_q^{2}[x,y]\left( \|x\|^{p}-\|y\|^{p} \right)^{2}- \alpha_p^{2}[x,y]\left( \|x\|^{q}-\|y\|^{q} \right)^{2}\\
=&\ \alpha^{2}[x,y]\left(\|x\|^{q}\|y\|^{q}\alpha_p^{2}[x,y]-\|x\|^{p}\|y\|^{p}\alpha_q^{2}[x,y]   \right),
\end{split}
\end{equation*}
which implies that inequality (\ref{apasolutno}) is valid if and only if holds the inequality
\begin{equation}\label{korijenje}
\alpha_p[x,y]\geq \|x\|^\frac{p-q}{2}\|y\|^{\frac{p-q}{2}}\alpha_q[x,y].
\end{equation}
Finally, by virtue of statement (III) in  Theorem A, inequality (\ref{korijenje}) holds for $|p|\geq |q|$, which means that (\ref{apasolutno}) holds for $\big| \frac{p}{q} \big|\geq 1$
in an inner product space.

Utilizing identity (\ref{identitet0}) once again, we will establish yet another inequality through equivalence with   (\ref{korijenje}).
We have already discussed that inequality (\ref{apasolutno}) holds in an arbitrary  normed linear space when $\frac{p}{q}\geq 1$. On the other hand, our next result shows significantly different behavior from (\ref{apasolutno}), since we will
 obtain a  characterization of an inner product space.  In order to establish the corresponding result, we will
keep in mind  the next two theorems due to Lorch and Ficken.

\begin{thmB}[Lorch \cite{lorch}]
Let  $X=\left(X, \left\Vert \cdot \right\Vert \right)$ be a  normed linear space. Then the following statements are mutually equivalent:
\begin{enumerate}
\item[(i)] For each $x,y\in X$ if $\|x\|=\|y\|$, then $\|x+y\|\leq \|\mu x+\mu^{-1} y\|$, for all $\mu \neq 0$.

\item[(ii)] For each $x,y\in X$ if $\|x+y\|\leq \|\mu x+\mu^{-1} y\|$, for all $\mu \neq 0$, then $\|x\|=\|y\|$.

\item[(iii)] $X=\left(X, \left\Vert \cdot \right\Vert \right)$ is an inner product space.
\end{enumerate}
\end{thmB}

\begin{thmC}[Ficken \cite{ficken}]\label{teorem ficken}
Let  $X=\left(X, \left\Vert \cdot \right\Vert \right)$ be a  normed linear space. Then the following statements are mutually equivalent:
\begin{enumerate}
\item[(i)] For each $x,y\in X$ if $\|x\|=\|y\|$, then $\|\lambda x+\mu y\|= \|\mu x+\lambda y\|$, for all $\lambda,\mu >0$.

\item[(ii)] For each $x,y\in X$ if $\|x\|=\|y\|$, then $\|\mu x+\mu^{-1} y\|= \|\mu^{-1} x+\mu y\|$, for all $\mu>0$.

\item[(iii)] $X=\left(X, \left\Vert \cdot \right\Vert \right)$ is an inner product space.
\end{enumerate}
\end{thmC}

Now, we are in position to state and prove  the main result of this section.

\begin{theorem}\label{glavni teorem}
Let $|p|\geq |q|$, $p\neq q$. Then, a normed linear space  $X=\left(X, \left\Vert \cdot \right\Vert \right)$ is an inner product space
if and only if the inequality
\begin{equation}
\label{hile3}{\alpha_p[x,y]}\geq \frac{{\|x\|^{p}+\|y\|^{p} }}{\|x\|^{q}+\|y\|^{q}  }\alpha_q[x,y]
\end{equation}
holds for all nonzero vectors $x,y\in X$.
\end{theorem}
\begin{proof}
Let $X$ be an inner product space.
We will show equivalence of relations (\ref{korijenje}) and (\ref{hile3}), by rewriting (\ref{identitet0}) in a more suitable form:
$$
\alpha_p^{2}[x,y]=\left( \|x\|^{p}+\|y\|^{p} \right)^{2}+\|x\|^{p}\|y\|^{p}(\alpha^{2}[x,y]-4).
$$
Therefore we have
$$
\frac{\alpha_p^{2}[x,y]}{\alpha_q^{2}[x,y]}=\frac{\left( \|x\|^{p}+\|y\|^{p} \right)^{2}+\|x\|^{p}\|y\|^{p}(\alpha^{2}[x,y]-4)}{\left( \|x\|^{q}+\|y\|^{q} \right)^{2}+\|x\|^{q}\|y\|^{q}(\alpha^{2}[x,y]-4)},
$$
which can be rewritten as
\begin{equation*}
\begin{split}
&\ \alpha_q^{2}[x,y]\left( \|x\|^{p}+\|y\|^{p} \right)^{2}-\alpha_p^{2}[x,y]\left( \|x\|^{q}+\|y\|^{q} \right)^{2}\\
=&\ \left(\alpha^{2}[x,y]-4\right)\left( \|x\|^{q}\|y\|^{q}\alpha_p^{2}[x,y]-\|x\|^{p}\|y\|^{p}\alpha_q^{2}[x,y]  \right).
\end{split}
\end{equation*}
Now, since $\alpha[x,y]\leq 2$, the previous identity implies that inequalities (\ref{korijenje}) and (\ref{hile3}) are equivalent for $\alpha[x,y]<2$.
It remains to prove (\ref{hile3}) when $\alpha[x,y]=2$, which is trivial. Namely, if $\alpha[x,y]=2$, then there exist $\lambda>0$ such that $x+\lambda y=0$.
Then, $\alpha_p[x,y]=(1+\lambda^{p})\|x\|^{p}=\|x\|^{p}+\|y\|^{p}$, and we have equality sign in (\ref{hile3}). In conclusion, inequality (\ref{hile3}) holds whenever $|p|\geq |q|$.

Now, let $X$ be a normed linear space satisfying inequality (\ref{hile3}). We prove that $X$ is an inner product space by considering three cases $|p|>|q|>0$,   $p=-q$, and $q=0$ separately.

{\bf Case 1.} Let $|p|>|q|>0$. From validity of inequality  (\ref{hile3}), it follows that holds the inequality
\begin{equation*}
{\alpha_p[x',y']}\geq \frac{{\|x'\|^{p}+\|y'\|^{p} }}{\|x'\|^{q}+\|y'\|^{q}  }\alpha_q[x',y'],
\end{equation*}
where $x'=\|x\|^{\frac{1}{p}-1}x$ and $y'=\|y\|^{\frac{1}{p}-1}y$. Consequently, since $\alpha_p[x',y']=\|x-y\|$ and
$\alpha_q[x',y']=\alpha_{q/p}[x,y]$, it follows validity of  the inequality
\begin{equation}\label{pomocna}
\alpha_r[x,y]\leq \frac{\|x\|^{r}+\|y\|^{r}}{\|x\|+\|y\|}\|x-y\|,
\end{equation}
where $0<|r|=\big| \frac{q}{p} \big|<1$ and $x,y\neq 0$. Now, let $x,y\in X$ be such that $\|x\|=\|y\|$ and let $\mu\neq 0$. According to Theorem B, it suffices to prove that
$\|x+y\|\leq \|\mu x+\mu^{-1}y\|$. If $\|x\|=\|y\|=0$, this relation holds trivially. Otherwise, applying inequality (\ref{pomocna}) to $x''=\mu^{r^{n}}x$ and $y''=-\mu^{-r^{n}}y$, where $\mu>0$, instead of $x$ and $y$ respectively,
we obtain
$$
\alpha_r[x'',y'']\leq \frac{\|x''\|^{r}+\|y''\|^{r}}{\|x''\|+\|y''\|}\|x''-y''\|.
$$
Moreover, since $\|x\|=\|y\|$, we have  that
\begin{equation*}
\begin{split}
\alpha_r[x'',y'']&=\left\| \mu^{r^{n}(r-1)}\mu^{r^{n}}\|x\|^{r-1}x +\mu^{-r^{n}(r-1)}\mu^{-r^{n}}\|y\|^{r-1}y  \right\|\\
&=\|x\|^{r-1}\left\| \mu^{r^{n+1}}x+\mu^{-r^{n+1}}y \right\|,
\end{split}
\end{equation*}
so the above inequality reduces to
\begin{equation}\label{pomoc2}
\left\| \mu^{r^{n+1}}x+\mu^{-r^{n+1}}y \right\|\leq \frac{\mu^{r^{n+1}}+\mu^{-r^{n+1}}}{\mu^{r^{n}}+\mu^{-r^{n}}}\left\| \mu^{r^{n}}x+\mu^{-r^{n}}y \right\|.
\end{equation}
Now, consider the function $f(x)=a^x+a^{-x}$, where $a>0$, $a\neq 1$. Clearly, $f$ is an  even function. In addition, since $f'(x)=\log a (a^x-a^{-x})$, it follows that $f$ is increasing for $x>0$,
while it is decreasing for $x<0$. Consequently, we have that $a^{r}+a^{-r}<a+a^{-1}$, for $|r|<1$. Hence, by putting $a=\mu^{r^{n}}$ in the last inequality, we obtain
${\mu^{r^{n+1}}+\mu^{-r^{n+1}}}\leq {\mu^{r^{n}}+\mu^{-r^{n}}}$, so  (\ref{pomoc2}) implies that
$$
\left\| \mu^{r^{n+1}}x+\mu^{-r^{n+1}}y \right\|\leq \left\| \mu^{r^{n}}x+\mu^{-r^{n}}y \right\|.
$$
In other words,   $a_n=\big\| \mu^{r^{n}}x+\mu^{-r^{n}}y \big\|$, $n\in \mathbb{N}\cup\{0\}$, is a  decreasing sequence  which implies that
$$
\|x+y\|=\lim_{n\rightarrow\infty} a_n \leq a_0=\|\mu x+\mu^{-1}y\|.
$$
In addition, if $\mu$ is negative, then by putting $-\mu$ in the last inequality, we obtain
$$
\|x+y\|\leq \|(-\mu) x+(-\mu)^{-1}y\|=\|\mu x+\mu^{-1}y\|.
$$
In conclusion, $X$ is an inner product space by Theorem B.

{\bf Case 2.} Let $p=-q$. Then, taking into account (\ref{pomocna}), it follows that  the inequality
$$
\left\|\|x\|^{-2}x-\|y\|^{-2}y  \right\|=\alpha_{-1}[x,y]\leq \frac{\|x\|^{-1}+\|y\|^{-1}}{\|x\|+\|y\|}\|x-y\|
$$
holds for all nonzero vectors $x,y\in X$. Clearly, the above relation is equivalent to
\begin{equation}\label{pomoc3}
\left\|\frac{\|y\|}{\|x\|}x-\frac{\|x\|}{\|y\|}y   \right\|\leq \|x-y\|.
\end{equation}
Now, assume that $x,y\in X$ are such that $\|x\|=\|y\|$ and let $\mu> 0$. Applying inequality (\ref{pomoc3}) to $\mu x$ and $-\mu^{-1}y$ instead of $x$ and $y$ respectively,
we obtain the inequality $\big\| \mu^{-1}x+\mu y \big\|\leq \big\| \mu x+\mu^{-1}y\big\|$. In the same way, inequality (\ref{pomoc3}) with  $\mu^{-1} x$ and $-\mu y$ instead of $x$ and $y$ respectively, yields the inequality
with the reversed sign, which implies that $\big\| \mu^{-1}x+\mu y \big\|= \big\| \mu x+\mu^{-1}y\big\|$. Therefore, Theorem C ensures that $X$ is an inner product space.

{\bf Case 3.} It remains to consider the case of $q=0$. In this setting, inequality (\ref{hile3}) reduces to
\begin{equation}\label{trebat ce}
\alpha[x,y]\leq \frac{2}{\|x\|^{p}+\|y\|^{p}}\alpha_p[x,y],
\end{equation}
which is covered by Theorem A,    statement (II), when dimension of $X$ is not less than 3. Alternatively, we can also exploit Theorem B which has no restriction on  dimension of a space.
Namely, let $x,y$ be nonzero vectors such that $\|x\|=\|y\|$
 and let $\mu>0$. Then, applying inequality (\ref{trebat ce}) to $\mu^{\frac{1}{p}} x$ and $-\mu^{-\frac{1}{p}}y$ instead of $x$ and $y$ respectively, as well as utilizing the arithmetic geometric
mean inequality, we obtain
$$
\|x+y\|\leq \frac{2}{\mu+\mu^{-1}}\|\mu x+\mu^{-1}y\|\leq \|\mu x+\mu^{-1}y\|.
$$
 Consequently, Theorem B implies that $X$ is an inner product space which completes the proof.
 \end{proof}

\begin{remark}
It should be noted  here that the proof of Theorem A relies on the symmetry of Birkoff-James orthogonality which is a characterization of inner pro\-duct spaces when dimension
of the corresponding space is at least three (for more details, see \cite{bj1,bj2,rocky}). If we would replace an arbitrary real parameter $r$ in statement (V) of Theorem A with a nonnegative parameter $r$,
then the statements (I)--(V) would be equivalent in any normed linear space, regardless of dimension. Namely, in that case, the corresponding equivalence could be established via Theorems B and C.
\end{remark}

\begin{remark}
If $q=1$, inequality (\ref{hile3}) reduces to
\begin{equation}
\label{hile4}
\alpha_p[x,y]\geq \frac{\|x\|^{p}+\|y\|^{p}}{\|x\|+\|y\|}\|x-y\|,
\end{equation}
where $|p|\geq 1$. Although inequality (\ref{hile4}) seems to be familiar to the Hile inequality (\ref{obicni hile}),  they show significantly different behavior.
Namely, the Hile inequality holds in every normed linear space, while (\ref{hile4}) characterizes an inner product space for $|p|>1$. On the other hand, if $q=0$,
inequality (\ref{hile3}) becomes
\begin{equation*}
\alpha_p[x,y]\geq \frac{\|x\|^{p}+\|y\|^{p}}{2}\alpha[x,y],
\end{equation*}
which coincides with the corresponding inequality in Theorem A.
\end{remark}

Theorem A provides characterizations of  inner product spaces expressed in terms of power means. Recall that a weighted power mean $M_r(x_1, x_2, \ldots , x_n)$ is defined by
$$
M_r(x_1, x_2, \ldots , x_n)=\bigg( \sum_{k=1}^n p_k x_k^{r} \bigg)^{\frac{1}{r}},
$$
where $\sum_{k=1}^n p_k=1$, $p_k>0$, and $(x_1, x_2, \ldots , x_n)$ is a positive $n$-tuple. The nonweighted means correspond to the setting $p_k=1/n$, $k=1,2,\ldots , n$. Recall that for $p=1, 0, -1$ we obtain respectively, the arithmetic, geometric and harmonic mean.
In addition, $M_{-\infty}(x_1, x_2, \ldots , x_n)=\min\{x_1, x_2, \ldots , x_n\}$ and $M_{\infty}(x_1, x_2, \ldots , x_n)=\max\{x_1, x_2, \ldots , x_n\}$.
The generalized mean inequality asserts that if $r<s$, then
\begin{equation}
\label{mean ineq}
M_r(x_1, x_2, \ldots , x_n)\leq M_s(x_1, x_2, \ldots , x_n).
\end{equation}
Inequality (\ref{mean ineq}) is true for real values of $r$ and $s$, as well as for positive and negative infinity values. For more details about the means inequalities,
the reader is referred to \cite{11}.

Taking into account the monotonicity property (\ref{mean ineq}),   Theorem A asserts that the  inequality
\begin{equation}\label{sredine}
\alpha_p[x,y]\geq \left(\frac{\|x\|^{pr}+\|y\|^{pr}}{2}\right)^{\frac{1}{r}}\alpha[x,y]
\end{equation}
characterizes an inner product space for every  $r\leq 1$. Recall, once again, that for $r=0$ the above inequality yields the estimate $\alpha_p[x,y]\geq\|x\|^{\frac{p}{2}}\|y\|^{\frac{p}{2}}\alpha[x,y]$,
due to the well-known limit value $\lim_{x\rightarrow 0}\big(\frac{a^{x}+b^{x}}{2}  \big)^{\frac{1}{x}}=\sqrt{ab}$.

\begin{remark}\label{deveti remark}
It was shown  in \cite{rocky} that inequality (\ref{sredine}) is the best possible in the sense that there is no exponent $r>1$ such that (\ref{sredine}) is valid for all nonzero vectors $x,y\in X$.
This can also be nicely seen by considering any two linearly dependent vectors $x,y$ such that $x+\lambda y=0$, $\lambda>0$. Then, $\alpha_p[x,y]=\|x\|^{p}+\|y\|^{p}$, $\alpha[x,y]=2$, so in this case (\ref{sredine})
reduces to
$$
\frac{\|x\|^{p}+\|y\|^{p}}{2}\geq \left(\frac{\|x\|^{pr}+\|y\|^{pr}}{2}\right)^{\frac{1}{r}},
$$
which is in contrast to (\ref{mean ineq}).
\end{remark}

According to Remark \ref{deveti remark},  inequality (\ref{sredine}) does not hold for $r>1$ in general. However, it is easy to describe pairs of vectors for which  inequality (\ref{sredine})
holds when $r=2$. The corresponding result relies on identity (\ref{identitet0}).

\begin{corollary}
Let $X=\left( X,\left\langle \cdot ,\cdot
\right\rangle \right) $ be an inner product space and let $x,y\in X$ be nonzero vectors. If $\alpha[x,y]\leq \sqrt{2}$, then holds the inequality
\begin{equation}
\label{kvadratik2}
\alpha_p[x,y]\geq \sqrt{\frac{\|x\|^{2p}+\|y\|^{2p}}{2}}\alpha[x,y],
\end{equation}
while for $\alpha[x,y]>\sqrt{2}$, the sign of inequality {\rm(\ref{kvadratik2})} is reversed.
\end{corollary}

\begin{proof}
Since $X$ is an inner product space, identity (\ref{identitet0}) is valid and it can be rewritten as
$$
\alpha_p^{2}[x,y]= \frac{\|x\|^{2p}+\|y\|^{2p}}{2}\alpha^{2}[x,y]+\frac{1}{2}\big( \|x\|^{p}- \|y\|^{p}\big)^{2}\big(2-\alpha^{2}[x,y]\big).
$$
Now, if  $\alpha[x,y]\leq \sqrt{2}$, then the second term on the right-hand side of the previous identity is nonnegative, which implies (\ref{kvadratik2}).
In the same way we conclude that for $\alpha[x,y]>\sqrt{2}$ holds the reverse in (\ref{kvadratik2}).
\end{proof}

Our next intention is to establish connection between mutual bounds for angular distance $\alpha[x,y]$ given by (\ref{obostranoao}), and the explicit formula for $\alpha[x,y]$ given by (\ref{zadva}),
 which is a characterization of an inner product space.
It is interesting that if (\ref{zadva}) is valid, then  the square of angular distance $\alpha[x,y]$ is equal to the product of mutual bounds in (\ref{obostranoao}). This fact allows us to
refine these bounds in an inner product space.

\begin{theorem}\label{tm popravi maligrandu}
Let $X=\left( X,\left\langle \cdot ,\cdot
\right\rangle \right) $ be an inner product space and let $r>0$. If $x,y\in X$ are nonzero vectors with $x\neq y$, then hold the  inequalities
\begin{equation}\label{obostranoao_popravi}
\begin{split}
&\left( \frac{1+c^{r}[x,y]}{2} \right)^{-\frac{1}{r}}\dfrac{\Vert x-y\Vert-|\Vert x\Vert - \Vert y\Vert|}{\min\{\Vert x\Vert,\Vert y\Vert\}}\\
\leq &\  \alpha[x,y]
\leq\left( \frac{1+c^{r}[x,y]}{2} \right)^{\frac{1}{r}} \dfrac{\Vert x-y\Vert+|\Vert x\Vert - \Vert y\Vert|}{\max\{\Vert x\Vert,\Vert y\Vert\}},
\end{split}
\end{equation}
where
\begin{equation*}
c[x,y]=\frac{\max\{\Vert x\Vert,\Vert y\Vert\}}{\min\{\Vert x\Vert,\Vert y\Vert\}}\cdot \frac{\Vert x-y\Vert-|\Vert x\Vert - \Vert y\Vert|}{\Vert x-y\Vert+|\Vert x\Vert - \Vert y\Vert|}<1.
\end{equation*}
\end{theorem}
\begin{proof}
Taking into account formula (\ref{zadva}), we have  that $\alpha^{2}[x,y]=AB$,
where
$$
A=\dfrac{\Vert x-y\Vert-|\Vert x\Vert - \Vert y\Vert|}{\min\{\Vert x\Vert,\Vert y\Vert\}}\quad \mathrm{and}  \quad B=\dfrac{\Vert x-y\Vert+|\Vert x\Vert - \Vert y\Vert|}{\max\{\Vert x\Vert,\Vert y\Vert\}}.
$$This means that $\alpha[x,y]$ is a  geometric mean of  $A$ and $B$.
Consequently, due to the  monotonicity property (\ref{mean ineq}), it follows that the  inequalities
$$
\left( \frac{A^{-r}+B^{-r}}{2}  \right)^{-\frac{1}{r}}\leq \alpha[x,y]\leq \left( \frac{A^{r}+B^{r}}{2}  \right)^{\frac{1}{r}}
$$
hold for every $r>0$. Now, the result follows by noting that $c[x,y]=\frac{A}{B}$.
\end{proof}

\begin{remark}\label{napomena za 1}
It follows from the construction that the  inequalities in (\ref{obostranoao_popravi}) provide refinements of mutual bounds in (\ref{obostranoao}) for each exponent $r>0$.
 More precisely, since $\alpha[x,y]$ is represented as the geometric mean of bounds in (\ref{obostranoao}), it follows by monotonicity property (\ref{mean ineq}) that the upper bound
 in (\ref{obostranoao_popravi}) is not greater than the upper bound in (\ref{obostranoao}). In the same way, the lower bound in (\ref{obostranoao_popravi}) is more accurate than the lower bound in 
 (\ref{obostranoao}). 
 In particular, if $r=1$, the  inequalities in (\ref{obostranoao_popravi})
reduce to  the following form:
\begin{equation}\label{za r 1}
\begin{split}
& \frac{2\big[ \left\Vert x-y\right\Vert^2-\left(\left\Vert x\right\Vert-\left\Vert y\right\Vert\right)^2\big]}{\|x-y\| (\|x\|+\|y\|)-(\|x\|-\|y\|)^{2}}\\
\leq &\ \alpha[x,y]\leq \frac{\|x-y\| (\|x\|+\|y\|)-(\|x\|-\|y\|)^{2}}{2\|x\|\|y\|}.
\end{split}
\end{equation}
\end{remark}

The second inequality in (\ref{za r 1}) can be exploited in obtaining a parametric family of characterizing relations
for an inner product space.

\begin{corollary}
Let $X=\left(X, \left\Vert \cdot \right\Vert \right)$ be a normed linear space, let $0<r\leq 1$, and let $c[x,y]$ be defined as in the statement of Theorem {\rm \ref{tm popravi maligrandu}}. Then, $X$ is an inner product space
if and only if the   inequality
\begin{equation}\label{karakter}
\alpha[x,y]
\leq\left( \frac{1+c^{r}[x,y]}{2} \right)^{\frac{1}{r}} \dfrac{\Vert x-y\Vert+|\Vert x\Vert - \Vert y\Vert|}{\max\{\Vert x\Vert,\Vert y\Vert\}}
\end{equation}
holds for all nonzero vectors $x,y\in X$ with $x\neq y$.
\end{corollary}
\begin{proof}
It suffices to prove that a normed linear space satisfying (\ref{karakter}) is an inner product space. Taking into account Theorem \ref{tm popravi maligrandu}, Remark \ref{napomena za 1}, property (\ref{mean ineq}) and the triangle inequality, we have
\begin{equation*}
\begin{split}
\alpha[x,y]
&\leq\left( \frac{1+c^{r}[x,y]}{2} \right)^{\frac{1}{r}} \dfrac{\Vert x-y\Vert+|\Vert x\Vert - \Vert y\Vert|}{\max\{\Vert x\Vert,\Vert y\Vert\}}\\
&\leq \frac{\|x-y\| (\|x\|+\|y\|)-(\|x\|-\|y\|)^{2}}{2\|x\|\|y\|}\\
&=\frac{\|x-y\| (\|x\|+\|y\|)^2-(\|x\|+\|y\|)(\|x\|-\|y\|)^{2}}{2\|x\|\|y\|(\|x\|+\|y\|)}\\
&\leq \frac{\big[(\|x\|+\|y\|)^2-(\|x\|-\|y\|)^2   \big]\|x-y\|}{2\|x\|\|y\|(\|x\|+\|y\|)}=\frac{2\|x-y\|}{\|x\|+\|y\|}.
\end{split}
\end{equation*}
Obviously, this proves our assertion due to characterization (\ref{stanley}) given by Kirk and Smiley.
\end{proof}

In order to conclude this paper, we give the  lower bound for the  $p$-angular distance expressed in terms of angular distance $\alpha[x,y]$, which is also consequence of
the right inequality in (\ref{za r 1}). The following lower bound for $\alpha_p[x,y]$  also gives a characterization of inner product spaces.

\begin{corollary}\label{pametniteorem}
Let $p\neq 0$. Then a normed linear space  $X=\left(X, \left\Vert \cdot \right\Vert \right)$, $\mathrm{dim}\ \! X\geq 3$, is an inner product space
if and only if the   inequality
\begin{equation}
\label{pametno}
\alpha_p[x,y]\geq \frac{2\|x\|^{p}\|y\|^{p}\alpha[x,y]+(\|x\|^{p}- \|y\|^{p})^{2}}{\|x\|^{p}+ \|y\|^{p}}
\end{equation}
holds for all nonzero vectors $x,y\in X$.
\end{corollary}
\begin{proof}
Let $X$ be an inner product space. Then, rewriting  the right inequality in (\ref{za r 1}) with
$x'=\|x\|^{p-1}x$, $y'=\|y\|^{p-1}y$ instead of $x$, $y$, respectively, we obtain the inequality
$$
\alpha[x,y]\leq\frac{\alpha_p[x,y](\|x\|^{p}+ \|y\|^{p})-(\|x\|^{p}- \|y\|^{p})^{2}}{2\|x\|^{p}\|y\|^{p}},
$$
due to homogeneity of $\alpha[x,y]$, that is,  $\alpha[x',y']=\alpha[x,y]$. Obviously, the last inequality is equivalent to (\ref{pametno}).

On the other hand, assume that $X$ is a normed linear space satisfying inequality (\ref{pametno}). Then we have
$$
\alpha_p[x,y]\geq \frac{2\|x\|^{p}\|y\|^{p}\alpha[x,y]}{\|x\|^{p}+ \|y\|^{p}}=\left( \frac{\|x\|^{-p}+ \|y\|^{-p}}{2} \right)^{-1}\alpha[x,y],
$$
so $X$ is an inner product space due to Theorem A, statement (V).
\end{proof}

\noindent {\bf Acknowledgment.} The authors would like to thank the  anonymous referee for some valuable comments and useful suggestions.

\vspace*{5mm}

\noindent  \noindent University of Zagreb, Faculty of Electrical Engineering and Computing,  Unska 3, 10000 Zagreb, CROATIA \\
E-mail: {\tt mario.krnic@fer.hr}

\vspace*{3mm}

\noindent Transilvania University of Bra\c{s}ov, Iuliu Maniu Street, No. 50, 500091, Bra\c{s}ov, ROMANIA \\
E-mail: {\tt minculeten@yahoo.com}


\begin{thebibliography}{99}

\bibitem{app2} M. Adamek, {\it Characterization of Inner Product Spaces by Strongly Schur-Convex Functions}, Results Math. {\bf 75}, 72 (2020), https://doi.org/10.1007/s00025-020-01197-1.

\bibitem{alrashed}
A.M. Al-Rashed, {\it Norm Inequalities and Characterizations of Inner Product Spaces}, J. Math. Anal. Appl. {\bf 176} (1993), 587--593.

\bibitem{alsina}
C. Alsina, M.S. Tomas, {\it On parallelogram areas in normed linear spaces}, Aequationes Math. {\bf 72}  (2006), 234--242.

\bibitem{amini-mia}
A. Amini-Harandi, M. Rahimi,  M. Rezaie, {\it Norm inequalities and characterizations of inner product spaces}, Math. Inequal. Appl. {\bf 21} (2018),  287--300.

\bibitem{amir}
D. Amir, {\it Characterizations of Inner Product Spaces}, Oper. Theory Adv. Appl. vol. 20, Birkh\"{a}user Verlag, Basel, 1986.

\bibitem{klarkson}
J.A. Clarkson, {\it Uniformly convex spaces}, Trans. Amer. Math. Soc. {\bf 40} (1936), 396--414.

\bibitem{dadipur}
F. Dadipour, M.S. Moslehian, {\it A characterization of inner product spaces related to the
p-angular distance}, J. Math. Anal. Appl. {\bf 371} (2010),  677--681.

\bibitem{bj1}
M.M. Day, {\it Some characterizations of inner product spaces}, Trans. Amer. Math. Soc. {\bf 62}  (1947), 320--337.

\bibitem{degen}
H. Dehghan,  {\it A characterization of inner product spaces related to the skew-angular distance}, Math. Notes {\bf 93} (2013), 556--560.

\bibitem{ukraine}
S.S. Dragomir, {\it New inequalities for the $p$-angular distance in normed spaces with applications}, Ukrainian Math. J. {\bf 67} (2015), 19--32.

\bibitem{dragomirMIA}
S.S. Dragomir, {\it Inequalities for the $p$-angular distance in normed linear spaces}, Math. Inequal. Appl. {\bf 12} (2009), 391--401.

\bibitem{dunkl}
C.F. Dunkl, K.S. Williams, {\it Mathematical notes: A simple norm inequality}, Amer. Math. Monthly {\bf 71} (1964), 53--54.


\bibitem{ficken}
F.A. Ficken, {\it Note on the existence of scalar products in normed linear spaces}, Ann. of Math. {\bf 45} (1944), 362--366.

\bibitem{hile} G.N. Hile, {\it Entire solutions of linear elliptic equations with Laplacian
principal part}, Pacific J. Math. {\bf 62} (1976), 127--140.


\bibitem{bj2}
R.C. James, {\it Inner products in normed linear spaces}, Bull. Amer. Math. Soc. (N.S.) {\bf 53} (1947), 559--566.

\bibitem{stenli}
W.A. Kirk, M.F. Smiley, {\it Mathematical notes: Another characterization of inner product spaces}, Amer. Math. Monthly {\bf 71} (1964), 890--891.

\bibitem{krnic}
M. Krni\' c, N. Minculete, {\it Characterizations of inner product spaces via angular distances  and Cauchy-Schwarz inequality},  Aequat. Math. (2020), https://doi.org/10.1007/s00010-020-00735-0.

\bibitem{lorch}
E.R. Lorch, {\it On certain implications which characterize Hilbert space}, Ann. of Math. {\bf 49} (1948), 523--532.

\bibitem{maligranda} L. Maligranda,  {\it Simple Norm Inequalities}, Amer. Math. Monthly {\bf 113} (2006), 256--260.

\bibitem{app1} D.S. Marinescu, M. Monea, M. Opincariu, M. Stroe, {\it A Characterization Of The Inner Product Spaces Involving Trigonometry}, Ann. Funct. Anal. {\bf 4} (2013), 109--113.


\bibitem{11}
D.S. Mitrinovi\' c, J.E. Pe\v cari\' c, A.M. Fink, {\it Classical
and New Inequalities in Analysis}, Kluwer Academic Publishers,
Dordrecht/Boston/London, 1993.

\bibitem{rasias}
M.S. Moslehian, J.M. Rassias, {\it A characterization of inner product spaces concerning an Euler-Lagrange identity}, Commun. Math. Anal. {\bf 8} (2010), 16--21.


\bibitem{rooin0}
J. Rooin, S. Rajabi, M.S. Moslehian, {\it $p$-angular distance orthogonality}, Aequationes Math. {\bf 94} (2020), 103--121.

\bibitem{rocky}
J. Rooin, S. Rajabi, M.S. Moslehian, {\it Extension of Dunkl-Williams inequality and characterizations of inner product spaces}, Rocky Mountain J. Math. {\bf 49} (2019), 2755--2777.


\bibitem{rooin}
J. Rooin, S. Habibzadeh, M.S. Moslehian, {\it Geometric aspects of $p$-angular and skew $p$-angular distances}, Tokyo J. Math. {\bf 41} (2018), 253--272.

\bibitem{vuvu}
S. Wu, C. He, G. Yang, {\it Orthogonalities, linear operators, and characterizations of inner product spaces}, Aequationes Math. {\bf 91} (2017), 969--978.
\end{thebibliography}
\end{document}